\documentclass[12pt, reqno]{amsart}
\usepackage{graphicx} 

\usepackage{hyperref}
\usepackage{amscd, amssymb, latexsym, amsmath, amscd, enumerate}
\usepackage{amstext, anysize, amsthm}
\usepackage[normalem]{ulem}
\usepackage[sc]{mathpazo} 
\usepackage{enumitem}
\usepackage{pst-node}
\usepackage{multirow}
\usepackage{booktabs, tabularx, ragged2e}
\marginsize{2.5cm}{2.5cm}{2cm}{2cm}
\usepackage{graphics}
\usepackage{color}
\usepackage[all]{xy}
\usepackage{tikz}
\usepackage{tikz-cd}
\usepackage[square, sort, numbers]{natbib}
\newtheorem{theorem}{Theorem}[section]

\newtheorem{corollary}[theorem] {Corollary}
\newtheorem{definition}[theorem]{Definition}

\newtheorem{lemma} [theorem]{Lemma}

\newtheorem{proposition}[theorem]{Proposition}
\newtheorem{remark}[theorem]{Remark}

\numberwithin{equation}{section}

\newcommand{\F}{\mathbb{F}}

\newcommand{\Tr}{{\rm Tr}}

\newcommand{\GL}{{\rm GL}}

\newcommand{\Hom}{{\rm Hom}}

\newcommand{\Ind}{{\rm Ind}}

\newcommand{\N}{{\rm N}}

\newcommand{\Res}{{\rm Res}}

\newcommand{\Irr}{{\rm Irr}}

\newcommand{\NS}{\mathrm{NS}}

\newcommand{\Conj}{{\rm Conj}}

\title[Product of conjugacy classes v/s tensor product]{Trilinear forms for \( \GL_2(\F_q)\) via conjugacy classes} 

\author{Ankita Parashar}
\address{Department of Mathematics, IIT Bombay, Mumbai, Maharashtra - 400076}
\email{ankitaparashar216@gmail.com}

\author{Shiv Prakash Patel}
\address{Department of Mathematics, IIT Delhi, Hauz Khas, New Delhi - 110016} 
\email{shivprakashpatel@gmail.com}


\begin{document}

\begin{abstract}
It is well known that for a finite group \(G\), the number of elements in $\Irr(G)$ the set of irreducible representations up to isomorphism is the same as the number of elements in $\Conj(G)$ the set of conjugacy classes in $G$ although there is no natural bijection between these two sets.
For $G=\GL_2(\F_q)$, where $\F_q$ is the finite field with $q$ elements, we give a bijection between $\Irr(G)$ and $\Conj(G)$, which behaves well with the tensor product of irreducible representations on one side and the product of conjugacy classes on the other.
In particular, we relate the dimension of the trilinear forms with the multiplicity of the identity element in the triple product of the corresponding conjugacy classes.
\end{abstract}


\maketitle

\section{Introduction}\label{sec:intro}

Let \(G\) be a group. Consider \(G\) as a subgroup of \(G \times G \times G\) via diagonal embedding \(G  \hookrightarrow G \times G \times G \). 
Let \( \pi_1, \pi_2, \pi_3\) be its irreducible representations of \(G\).
The vector space of \(G\)-invariant linear forms \( \Hom_{G} (\pi_1 \otimes \pi_2 \otimes \pi_3, \mathbb{1}) \) is called the space of trilinear forms on  \( \pi_1, \pi_2, \pi_3\); where \( \mathbb{1} \) denotes the trivial representation of \(G\).
For a local field \(F\), the trilinear forms on representations of \(G=\GL_2(F)\) is deeply rooted with the arithmetic invariants attached to these representations via local Langlands correspondence \cite{Prasad-90, Prasad-07}. 
The dimension of the space of trilinear forms on representations of \(\GL_2(F)\) is less than or equal to one.
The trilinear forms on representations of \(\GL_2(F)\) are also related to the trilinear forms of representations of \(D^{\times}\), where \(D\) is the quaternion division algebra over \(F\), via Jacquet-Langlands correspondence \cite{Prasad-90}.

Let \( \F_q\) be a finite field with \(q\) elements and odd characteristic.
In this paper, we discuss the trilinear forms on representations of \(\GL_2(\F_q)\) and relate it with triple product of the conjugacy classes in \(\GL_2(\F_q)\). 
Let us recall the following result due to Prasad, see Example \cite[Example 2.5]{Prasad-2020} for \({\rm PGL}_2(\F_q)\), adapted for \(\GL_2(\F_q)\).
\begin{theorem}
Let \( \pi_1, \pi_2, \pi_3\) be irreducible representations of \(\GL_2(\F_q)\) such that the product of their central characters is trivial.    
Assume that all \(\pi_i\)'s are either principal series representations or cuspidal representations.
Then, except in a few exceptional cases, we have  \(\dim \Hom_{\GL_2(\F_q)} (\pi_1 \otimes \pi_2 \otimes \pi_3, \mathbb{1}) =1 \).
If all \(\pi_i\)'s are principal series representations (respectively, cuspidal representations) in the exceptional case, the dimension of the space of trilinear forms is 2 (respectively, 0). 
\end{theorem}

In particular, the dimension of the space of trilinear forms for \(\GL_2(\F_q)\) need not be less than or equal to one, it can be two in some exceptional cases. These exceptional cases are described in Section 6.

For a finite group \(G\), let \( \Irr(G)\) and \(\Conj(G)\) denote the set of irreducible representations of \(G\) up to isomorphism  and the set of conjugacy classes of \(G\), respectively.
A fundamental result for finite groups \(G\) is \( |\Irr(G)| =   |\Conj(G)| \).
Thus, there exists a bijection between \( \Irr(G) \) and \(\Conj(G)\). For an arbitrary finite group, there is no canonical bijection between \(\Irr(G)\) and \(\Conj(G)\).
For some families of groups, various connections between conjugacy classes and representations are already known.
For example, the representations of symmetric groups and their connection to conjugacy classes via partitions has been extensively studied, for example, see \cite{James-Kerber}.
In addition, Young diagrams provide a combinatorial framework that plays a central role in this correspondence.
Similar phenomena arise for finite groups of Lie type, where the work of Green \cite{Green1955} and Deligne--Lusztig \cite{Deligne-Lusztig-1976} provides powerful methods for studying their characters.

In this paper, for \(G=\GL_2(\F_q)\), we will give a bijection (see Section \ref{sec:6}) between 
\[
\Psi : \Irr(G) \rightarrow \Conj(G)
\] 
which behaves well in the sense explained below.

Let \(C_g\) denote the conjugacy class such that \(g \in C_{g}\) for \(g \in G\).
In analogy to the tensor product \(\pi_1 \otimes \pi_2\) for \(\pi_1, \pi_2 \in \Irr(G)\),  we consider the product \(C_{g_1} \cdot C_{g_2}\) of two conjugacy classes \(C_{g_1},  C_{g_2} \in \Conj(G) \).
Note that an element in the product \(C_{g_1} \cdot C_{g_2}\) may be realized in many ways and hence we consider it as a multiset. It can easily be seen that if \(g \in C_{g_1} \cdot C_{g_2}\) then \( C_{g} \subset C_{g_1} \cdot C_{g_2}\) and every element in \(C_{g}\) appears the same number of times which we call the multiplicity of \(g\) or the multiplicity of \(C_{g}\).
The underlying philosophy of this article is that if \(\pi_1, \pi_2 \in \Irr(G)\) and \( \Psi(\pi_i)=C_{g_i} \in \Conj(G)\) for \(i=1,2\) then \(\pi_1 \otimes \pi_2\) is {\it determined} by the multiset \( C_{g_1} \cdot C_{g_2} \) and vice-versa.
More precisely, if \(\pi \in \Irr(G)\) appears in \(\pi_1 \otimes \pi_2\) with multiplicity, say \(m(\pi)\) and  \( \Psi(\pi)=C_{g} \in \Conj(G)\) appears in \( C_{g_1} \cdot C_{g_2}\) with multiplicity \(m(C_g)\), then \(m(\pi)\) and \(m(C_g)\) determine each other.
Thus, this correspondence \(\Psi\) links the algebraic structure of representations with the combinatorial structure of conjugacy classes. 
We note that the choice of \(\Psi\) is not unique, and therefore the above phenomenon depends on \(\Psi\). 

Guided by this philosophy, we consider trilinear forms and relate them to the triple product of conjugacy classes.
We need some notation to state the main theorem.
For \( \pi_1, \pi_2, \pi_2 \in \Irr(G)\), let 
\[
m(\mathbb{1}, \pi_1 \otimes \pi_2 \otimes \pi_3) = \dim \Hom_{G}( \pi_1 \otimes \pi_2 \otimes \pi_3, \mathbb{1}),
\]
be the dimension of \(G\)-invariant trilinear forms on \(\pi_1, \pi_2, \pi_3\) or the multiplicity of \( \mathbb{1}\) the trivial representation in \(\pi_1 \otimes \pi_2 \otimes \pi_3\).
For \( g_1, g_2, g_3 \in G\), let 
\[
m\bigl(I,C_{g_1}\cdot C_{g_2}\cdot C_{g_3}\bigr):=\bigl|\{(x_1,x_2,x_3)\in C_{g_1}\times C_{g_2}\times C_{g_3} :  x_1x_2x_3=I\}\bigr|,
\]
be the multiplicity of \(I\) the identity matrix in the multiset \(C_{g_1} \cdot C_{g_2} \cdot C_{g_3}\) the triple product of conjugacy classes.
Moreover, one defines exceptional triples \(\{\pi_1,\pi_2,\pi_3\}\) and \(\{ C_{g_1}, C_{g_2}, C_{g_3} \} \), see Definition~\ref{exceptional}. 
We put \(\mathcal{E}_{ss}=1\) (respectively, \(\mathcal{E}_{ns}=1\)) if all \(\pi_i\) are principal series (respectively, cuspidal) representations and \(\{\pi_1,\pi_2,\pi_3\}\) is exceptional, and \(\mathcal{E}_{ss}=0\) (respectively, \(\mathcal{E}_{ns}=0\)) otherwise.
Now, we state the main result of the paper as follows.
\begin{theorem} \label{intro: main theorem} 
For \(G=\GL_2(\F_q)\), there exists a bijection 
\(\Psi:\Irr(G)\longrightarrow\Conj(G)\) satisfying the following properties.
Let \( \pi_1, \pi_2, \pi_2 \in \Irr(G)\) be irreducible Deligne-Lusztig representations, i.e. either irreducible principal series representations or cuspidal representations such that the product of their central characters is trivial.
Let \( \Psi(\pi_i) =C_{g_i} \in \Conj(G)\) for \(i=1,2,3\). Then we have the following:
\begin{enumerate}
\item 
\(\{\pi_1,\pi_2,\pi_3\}\) is exceptional if and only if \(\{C_{g_1},C_{g_2},C_{g_3}\}\) is exceptional.
\item \(
m\bigl(I,C_{g_1}\cdot C_{g_2}\cdot C_{g_3}\bigr)
=q(q^2-1)\,m\bigl(\mathbb{1},\pi_1\otimes\pi_2\otimes\pi_3\bigr)+q(q+1)\,\mathcal{E}_{ss}+q(q-1)\,\mathcal{E}_{ns}.
\)
\end{enumerate}
\end{theorem}
The following table summarizes the two multiplicities that appear in the above theorem.
\[
\begin{array}{|l|c|c|}
\hline
\pi_1,\pi_2,\pi_3
& m\bigl(\mathbb{1},\pi_1\otimes\pi_2\otimes\pi_3\bigr)
& m\bigl(I,C_{g_1}\cdot C_{g_2}\cdot C_{g_3}\bigr)\\ \hline
\text{Principal series, exceptional} & 2 & (2q-1)(q^2+q)\\
\text{Cuspidal, exceptional} & 0 & q^2-q\\
\text{Remaining cases} & 1 & q^3-q \\
\hline
\end{array}
\]

\begin{remark}
An analogous result holds when the three representations are not Deligne-Lusztig representations; however, its verification requires a case-by-case analysis.
We do not write many cases keeping in mind the simplicity of the above statement conveying the above mentioned philosophy.    
\end{remark}
\begin{remark}
There are many choices of the map \(\Psi\), which will be clear by its construction. It is important to note that the main theorem  does not depend on the choice of \(\Psi\).
\end{remark}
We prove this theorem by explicitly computing the multiplicities \( m(\mathbb{1},\pi_1\otimes\pi_2\otimes\pi_3)\) and \(m(I,C_{g_1} \cdot C_{g_2} \cdot C_{g_3})\).
The computation of  \( m(\mathbb{1}, \pi_1 \otimes \pi_2 \otimes \pi_3)\) follows from the decomposition of the tensor products of two irreducible representations obtained in \cite[Theorem 1.1]{Archita-Hassain} or by using characters of the corresponding representations as in \cite{G.Kaur}. 
On the other hand, the computation of \(m(I,C_{g_1} \cdot C_{g_2} \cdot C_{g_3})\) is one of the main parts of the proof of the above theorem.
 
Having calculated \(m(I,C_{g_1} \cdot C_{g_2} \cdot C_{g_3})\), we observe that it is always \(\geq 1\) if \(\det(g_1 g_2 g_3)=1\) and \(g_i\)'s are not scalar, giving a positive response to a well known question in Lie groups, known as the multiplicative Horn's problem or Deligne-Simpson problem \cite{Simpson-1990, Arielle-2025}, see Remark \ref{Horn's problem}.

The following additional feature of the formula in Theorem \ref{intro: main theorem} is worth emphasizing, which seems to have a natural generalization for groups of higher rank. 
First, let \(g_1,g_2,g_3\in G \) be semisimple with \(\det(g_1g_2g_3)=1\). 
If \(\{C_{g_1},C_{g_2},C_{g_3}\}\) is not exceptional, then
\(
m(I,\,C_{g_1}\cdot C_{g_2}\cdot C_{g_3})=q^3-q=|\mathrm{PGL}_2(\F_q)|.
\)
We expect an analogous statement for \(\GL_n(\F_q)\), i.e. for $g_1,g_2,g_3\in \GL_n(\F_q)$ with \(\det(g_1 g_2 g_3)=1\), in most of the cases, we would have
\[
m(I,\,C_{g_1}\cdot C_{g_2}\cdot C_{g_3})=|\mathrm{PGL}_n(\F_q)|=q^{n^2-1}\prod_{i=2}^{n}\bigl(1-q^{-i}\bigr).
\] 
Second, \( m(\mathbb{1},\pi_1\otimes\pi_2\otimes\pi_3)\) is the coefficient of \(q^3\) in \(m(I,C_{g_1}\cdot C_{g_2}\cdot C_{g_3})\).
We expect an analogous statement for \(\GL_n(\F_q)\), i.e.
\(m(\mathbb{1},\pi_1\otimes\pi_2\otimes\pi_3)\) should be the coefficient of \(q^{n^2-1}\) in \(m(I,C_{g_1}\cdot C_{g_2}\cdot C_{g_3})\).

The paper is organized in the following way.
In Section \ref{sec:2}, we introduce several notations and definitions that are relevant for the rest of the paper.
We describe products of relevant conjugacy classes in Section \ref{Sec:4}. 
In Section \ref{sec:6}, we construct the bijection \(\Psi\) between \(\Irr(\GL_2(\F_q))\) and \(\Conj(\GL_2(\F_q))\).
Finally, in Section \ref{Sec-5} we establish the main theorem.

\section{Preliminaries}\label{sec:2}
This section introduces the definitions and notation used throughout the paper.

Let \(\F_q\) be a finite field with \(q\) elements of characteristic \(p\) an odd prime.
Let $\delta$ be a square free element in $\F_q$.
Let $E$ be either $E\cong \F_q\oplus \F_q$ or $E\cong \F_{q^2}$.
Then \(\F_q \hookrightarrow E\) with the usual embedding when \(E=\F_{q^2}\) and
\(x \mapsto (x,x)\) when \(E= \F_q \oplus \F_q\).
Then we have a non-trivial automorphism $\sigma$ of $E$ as follows:
\[
z \mapsto \sigma(z) := z^{\sigma}=\left\{ \begin{array}{ll}
   (y,x) & \text{ for } z=(x,y)\in \F_q\oplus \F_q,
   \\
  z^q & \text{ for } z\in \F_{q^2}.
\end{array} \right.
\]

Then ${\mathrm {Aut}}_{\F_{q}}(E)$ the group of automorphisms of \(E\) over \(\F_q\) equals $\{1,\sigma\}$.
Clearly, $\F_q = E^{\sigma}$ the set of \(\sigma\)-fixed elements of \(E\).

\begin{definition}\label{Norm and trace Maps}
Let $ E $ be either $ \mathbb{F}_{q^2} $ or $ \mathbb{F}_q \oplus \mathbb{F}_q $. 
\begin{enumerate}
\item The norm map $\mathrm{N} : E \to \mathbb{F}_q $ is defined as $\mathrm{N}(z)=z\cdot z^{\sigma} $ for all $z\in E$.
\item The trace map $\mathrm{Tr} : E \to \mathbb{F}_q $ is defined as $\mathrm{Tr}(z)=z+z^{\sigma}$ for all $z\in E$.
\end{enumerate}
\end{definition}

The following is an easy observation.
\begin{lemma}\label{Norm and trace maps}
Let $ E $ be either $\mathbb{F}_{q^2} $  or $ E \cong \mathbb{F}_q \oplus \mathbb{F}_q $. 
Then 
\begin{enumerate}
\item For $ a \in \mathbb{F}_q^\times $,
\(
| \mathrm{N}^{-1}(a)| = \left\{ \begin{array}{ll}
        q+1  & \text{ if } E =\F_{q^2}, \\
         q-1 & \text{ if } E = \F_{q} \oplus \F_{q}.
\end{array} \right.
\)
\item \( \quad 
| \mathrm{N}^{-1}(0)| = \left\{ \begin{array}{ll}
        1  & \text{ if } E =\F_{q^2}, \\
         2q-1 & \text{ if } E = \F_{q} \oplus \F_{q}.
\end{array} \right.
\)
\end{enumerate}
\end{lemma}

We now fix an embedding of \( E \) in \( M_2(\F_q) \) and hence \(E^{\times} \hookrightarrow \GL_2(\F_q)\) and then consider elements of \(E^{\times}\) as elements of \(\GL_2(\F_q)\).

\begin{enumerate}
\item If \( E \cong \F_q \oplus \F_q \), then \( E \hookrightarrow  M_2(\F_q) \) is given by
\((a, b) \mapsto \begin{pmatrix}
a & 0 \\
0 & b    \end{pmatrix} \).
\item If \( E = \F_{q^2} = \F_q(\sqrt{\delta}) \), where $\delta\in \F_{q}^\times\smallsetminus (\F_{q}^\times)^2$, then
$E\hookrightarrow M_2(\F_q)$ is given by
\( a + b\sqrt{\delta} \mapsto \begin{pmatrix}
        a & b\delta \\
        b & a
    \end{pmatrix} \).
\end{enumerate}

For $g\in \GL_2(\F_q)$, we write $C_g$ to denote the conjugacy class of \(g\).
The following is an easy observation.
\begin{lemma}\label{exact two elements of a class in the embedding}
If $g \in E^{\times}$,
then $C_{g} \cap E^{\times}=[g]_{\sigma}:=\{g, g^{\sigma}\}
$.
\end{lemma}

Recall that every element in \(\GL_2(\F_q)\) is either central or split non-semisimple or  split semisimple or non-split semisimple. 
We write \({\mathrm{Z, U, SS, NS}}\) for the subsets of \(\GL_2(\F_{q})\) containing elements from center, split non-semisimple, non-central split semisimple, non-split semisimple respectively.
We now summarize the conjugacy classes of $\GL_2(\F_q)$ in the following lemma, 
 see \cite{Fulton-Harris} for more details.

\begin{lemma}\label{conjugacy classes of GL2}
A complete set of representatives for the conjugacy classes of $\GL_2(\F_q)$ is given by
\(
Z\sqcup[\mathrm U]\sqcup[\mathrm{SS}]\sqcup[\mathrm{NS}],
\)
where
\begin{enumerate}
\item 
\(
Z =\{aI:a\in\F_q^\times\}\),
\item 
\(
[\mathrm U]=\left\{\begin{pmatrix}
a & 1\\
0 & a
\end{pmatrix}:a\in\F_q^\times\right\}\),
\item 
For \(E= \F_{q} \oplus \F_{q}\),  we have 
\(E^{\times} \smallsetminus \F_{q}^{\times} \twoheadrightarrow [\mathrm{SS}]\)  with fibers \( [x]_{\sigma}\) for \(x \in E^{\times} \smallsetminus \F_{q}^{\times}
\),
\item
For  \(E= \F_{q^2}\), we have 
\(E^{\times} \smallsetminus \F_{q}^{\times} \twoheadrightarrow [\mathrm{NS}]\) with fibers \([x]_{\sigma}\) for \( x \in E^{\times} \smallsetminus \F_{q}^{\times}\).
\end{enumerate}
\end{lemma}

\section{Products of semisimple conjugacy classes}\label{Sec:4}

In this section, we describe the product of conjugacy classes of semisimple elements. 
First, we recall some definitions and simple observations that will be useful in what follows.

For a finite group \(G\) and \(g \in G\), we write \(C_{g}\) for the conjugacy class of \(g\).
For \(g_1, g_2 \in G\), we consider the product \(C_{g_1} \cdot C_{g_2} := \{ xy \mid x \in C_{g_1}, y \in C_{g_2} \}\) as a multiset as one element may be realized in more than one way, i.e. for \(g \in C_{g_1} \cdot C_{g_2}\) we may have \(g=xy = x' y'\) for \(x,x' \in C_{g_1}\) and \(y, y' \in C_{g_2}\) with \(x \neq x'\) and \(y \neq y'\).
For \(g \in G\), the number of times \(g\) appears in \(C_{g_1} \cdot C_{g_2}\) is called its multiplicity. 
More generally, for a multiset \(X\) and \(x \in X\), the number of times \(x\) appears in \(X\) is called the multiplicity of \(x\) in \(X\), which is denoted by \(m(x;X)\).
In particular, the multiplicity of \(g \in C_{g_1} \cdot C_{g_2}\) equals \( | \{ (x,y) \in C_{g_1} \times C_{g_2} : xy=g\}|\).
The following is an easy observation.
\begin{lemma}
Let \(G\) be a finite group  and \(g_1, g_2 \in G\). Then 
\begin{enumerate}
\item \( C_{g_1} \cdot C_{g_2} = C_{g_2} \cdot C_{g_1}\).
\item \(g \in C_{g_1} \cdot C_{g_2} \implies C_{g} \subset C_{g_1} \cdot C_{g_2}\).
\item If \(g, g' \in G\) are conjugate then their multiplicities in \(C_{g_1} \cdot C_{g_2}\) are the same.
\end{enumerate}
\end{lemma}

\begin{definition}
For a finite group \(G\), let \(g, g_1, g_2 \in G\). Then any two elements in \(C_{g}\) have the same multiplicity in \(C_{g_1} \cdot C_{g_2}\), which we call the multiplicity of \(C_{g}\) in \(C_{g_1} \cdot C_{g_2}\) and denote it by \( m(C_g;g_1,g_2)\).
\end{definition}

\begin{definition}
For $x\in \F_q$,
define $T(x) := \{g\in \GL_2(\F_q) \mid \mathrm{Tr}(g)=x\}$.   
\end{definition}

\begin{lemma}\label{m(x) is independent of the representative g_2}
Let $g_1,g_2\in \GL_2(\F_q)$ and \(x \in \F_{q}\).
Then for any $g'_2\in C_{g_2}$,
\[
 |\left(C_{g_1}\cdot g_2\right)\cap T(x)|=|\left(C_{g_1}\cdot g'_2\right)\cap T(x)|. 
\]
\end{lemma}

\begin{proof}
Since \( g_2' \in C_{g_2} \), there exists \( h \in \GL_2(\F_q) \) such that $g_2' = h^{-1} g_2 h$.
Then we have a bijection 
\[
\phi: C_{g_1} \cdot g_2 \longrightarrow C_{g_1} \cdot g_2'
\quad \text{given by} \quad
\phi(y) := h^{-1} y h.
\]
Moreover, the trace map is invariant under conjugation, we have  $y \in T(x)$ if and only if  $\phi(y) \in T(x)$.
This proves the lemma.
\end{proof}

\begin{definition}
For $g_1,g_2\in \GL_2(\F_q)$ and \(x \in \F_q\), define $t(x; g_1, g_2):=|T(x) \cap C_{g_1} \cdot g_2| $.
\end{definition}
    
\begin{lemma}\label{formula count for semisimple elements}
Let $g_1,g_2\in \GL_2(\F_q)$ and $x\in \F_q$. 
Assume that \(
C_h \subset (C_{g_1}\cdot C_{g_2}) \cap T(x) \), 
for some $h\in \mathrm{SS}\cup \mathrm{NS}$.
Then
\begin{equation*}\label{formula for multiplicity count}
m(C_h;g_1,g_2)=
\dfrac{|C_{g_2}|}{|C_h|} \cdot t(x;g_1,g_2).
\end{equation*}
\end{lemma}

\begin{proof}
We prove this lemma by computing \(| (C_{g_1} \cdot C_{g_2}) \cap T(x) |\) in two different ways.
First, note that as a multi set we have
\( C_{g_1} \cdot C_{g_2} = \bigsqcup_{g \in C_{g_2}} C_{g_1}\cdot g \) and therefore
\[
(C_{g_1} \cdot C_{g_2}) \cap T(x) = \bigsqcup_{g \in C_{g_2}} (C_{g_1}\cdot g) \cap T(x). 
\]
Since every element in $ C_{g_1} \cdot g $ has multiplicity one, it follows that every element in $ C_{g_1} \cdot C_{g_2} $ has multiplicity $ |C_{g_2}| $.
Then we have
\[
| (C_{g_1} \cdot C_{g_2}) \cap T(x) | = t(x; g_1, g_2)  \cdot |C_{g_2}|. 
\]

On the other hand, the conjugacy class of  \(h \in \mathrm{SS}\cup \mathrm{NS}\) is determined by its trace and determinant.
Since \(C_h \subset (C_{g_1}\cdot C_{g_2}) \cap T(x)\) we get
\( \det(h) = \det(g_1 g_2)\) and \(\Tr(h)=x\). Thus \(C_h\) is the only conjugacy class in \(C_{g_1} \cdot C_{g_2}\) with \(\Tr(h)=x\) and hence
\[
| (C_{g_1} \cdot C_{g_2}) \cap T(x) | = |C_{h}| \cdot m(C_{h};g_1,g_2). 
\qedhere
\]
\end{proof}

\begin{remark}\label{trace and determinant determine the semisimple classes}
If \(h \in Z \cup {\mathrm U}\) then \(C_h\) is not determined by its trace and determinant. 
Then the argument in the proof of above lemma does not work for elements in \( Z \cup {\mathrm U}\).
However, one can write a formula for \(m(C_h; g_1, g_2)\) which we skip writing as it will not be needed.
\end{remark}

\begin{proposition}\label{main theorem-1-4}
\begin{enumerate}
\item
Let \( g_1, g_2 \in \mathrm{SS}\). 
Then for \(x \in \F_q\), we have
\[
 t(x; g_1, g_2) = 
\begin{cases}
        2q - 1 & \text{if } x \in \{ \mathrm{Tr}(g_1 g_2), \mathrm{Tr}(g_1 g_2^{\sigma}) \}, \\
        q - 1 & \text{otherwise.}
\end{cases}
\]
\item Let \( g_1, g_2 \in \mathrm{NS} \).
Then for \(x \in \F_q\), we have
\[
t(x; g_1, g_2)  = 
\begin{cases}
        1 & \text{if } x \in \{ \mathrm{Tr}(g_1 g_2), \mathrm{Tr}(g_1 g_2^{\sigma}) \}, \\
        q+1 & \text{otherwise.}
\end{cases}
\]
\item  Let \( g_1 \in \mathrm{SS} \), \( g_2 \in \mathrm{NS} \).
Then 
\( 
t(x; g_1, g_2)=q+1
\) for all \(x\in \F_q\).
\end{enumerate}
\end{proposition}

\begin{proof}
Let \(g_1 \in \mathrm{NS} \cup \mathrm{SS}\). 
Then
\[
J=
\begin{pmatrix}
A & B\\
C & D
\end{pmatrix} \in C_{g_1} \iff A+D = \mathrm{Tr}(g_1) \quad \text{and} \quad AD-BC = \det(g_1).
\]
Now, for each $x\in \F_q$, we calculate $t(x;g_1,g_2)$ which is 
the number of elements in $C_{g_1}\cdot g_2$ with trace  \(x\).
Equivalently, 
we need to solve the following system of equations in $A,B,C,D\in \F_q$,
\begin{equation}\label{main-eq}
A+D=\mathrm{Tr}(g_1),\quad
AD-BC=\det(g_1),\quad 
\mathrm{Tr}(Jg_2)=x.
\end{equation}
Therefore, 
\begin{equation}\label{count for J}
t(x;g_1,g_2) =|\{J\in C_{g_1}:\eqref{main-eq}\text{ holds}\}|.
\end{equation}
Using Lemma \ref{m(x) is independent of the representative g_2}, the number \( t(x;g_1,g_2)\) does not depend on the choice of representative for $C_{g_1}, C_{g_2}$. Therefore, we take suitable representatives $g_1$ and $g_2$.

First, we consider $g_1,g_2 \in \mathrm{SS}$.
Take $g_i=\begin{pmatrix}
    a_i & 0\\
    0 & d_i
\end{pmatrix} \in \GL_2(\F_{q})$ for $i=1,2$ with \(a_i \neq d_i\).
Then the Equations in \eqref{main-eq} are equivalent to
\begin{align*}
A =\dfrac{x-(a_1+d_1)d_2}{(a_2-d_2)},~D=\dfrac{a_2(a_1+d_1)-x}{(a_2-d_2)},
\end{align*}
\begin{equation} \label{BC=0}
BC=
\dfrac{(x-(a_1a_2+d_1d_2))(x-(a_1d_2+a_2d_1))}{(a_2-d_2)^2}.
\end{equation}
Now
\(
BC=0 \iff x\in \{a_1a_2+d_1d_2, a_1d_2+a_2d_1\},~i.e.~x\in \{ \mathrm{Tr}(g_1g_2),\mathrm{Tr}(g_1g_2^{\sigma})\}.
\)
Hence, using \eqref{count for J} together with Lemma \ref{Norm and trace maps}, we get 
\[
 t(x;g_1,g_2) =
|\{ (B,C): \eqref{BC=0} \text{ holds}\}|=
\begin{cases}
2q-1 & \text{if } x \in \{\mathrm{Tr}(g_1g_2),\, \mathrm{Tr}(g_1g_2^{\sigma})\},\\
q-1 & \text{otherwise}.
\end{cases}
\]
This completes the proof of part (1) of the proposition.
The proofs for parts (2) and (3) are similar and we skip writing.
\end{proof}

Now, we describe the multiplicities $m(C_g;g_1,g_2)$ of $C_{g}$ in 
\(
C_{g_1}\cdot C_{g_2}
\) for $g,g_1,g_2\in \NS$ and $g,g_1,g_2\in \mathrm{SS}$.
Let $E=\F_q\oplus \F_q$ or $E=\F_{q^2}$.
For $g_1,g_2\in E$, we define 
\[
K:= \{ g_1g_2,~g_1 g_2^{\sigma},~ g_1^{\sigma}g_2,~ g_1^{\sigma} g_2^{\sigma}\}.
\]

\begin{lemma}\label{lemma:no-scalars-mixed-types}
Let \(g_1,g_2\in G=\GL_2(\mathbb{F}_q)\) and 
\(X\in\{Z,\mathrm{U},\mathrm{SS},\mathrm{NS}\}\).
\begin{enumerate}
\item 
If 
\(
g_1\in X,~g_2\in G\smallsetminus X,
\)
then
\(
(C_{g_1}\cdot C_{g_2})\cap Z=\emptyset
\).
\item For \( \lambda I \in Z\) and \(C_{g_1}\) with \(g_1 \in X\), there exists a unique \(C_{g_2}\) with \(g_2 \in X\) such that \(\lambda I\in C_{g_1}\cdot C_{g_2}\).
\item If $\lambda I\in C_{g_1}\cdot C_{g_2}$, then $m(C_{\lambda I};g_1,g_2)=|C_{g_1}|=|C_{g_2}|$.
\end{enumerate}
\end{lemma}

\begin{lemma}\label{scalar in split and non-semisimple matrices}
Let $g_1,g_2\in E^\times$.
Then 
\(
(C_{g_1}\cdot C_{g_2}) \cap Z \subset K.
\)
In particular, \(I \in C_{g_1} \cdot C_{g_2} \iff I \in K\).
    
\end{lemma}

\begin{proof}
It is clear that $\lambda I\in C_{g_1}\cdot C_{g_2}\iff \lambda g_1^{-1}\in C_{g_2}$. 
Since $g_1,g_2\in E^\times$, 
we get $\lambda g_1^{-1}\in E^\times \cap C_{g_2} = \{ g_{2}, ~g_{2}^{\sigma} \}$ by Lemma \ref{exact two elements of a class in the embedding}.
Now, the lemma follows from $C_{g}=C_{g^{\sigma}}$.
\end{proof}

\begin{proposition}\label{proposition 3.11-1,2,3}
    \begin{enumerate}
        \item Let \(E=\F_q\oplus \F_q\) and $g_1, g_2 \in E^\times\smallsetminus \F_{q}^\times$. 
Let $g \in \mathrm{SS}\cup K$ with $\det(g) = \det(g_1 g_2)$. 
Then
$$
m(C_g; g_1, g_2) = \begin{cases} 
 q^2+q & \text{if } g \in K\cap Z,\\
2q-1 & \text{if } g \in \mathrm{SS}\cap K,
\\
q-1 & \text{if } g \in \mathrm{SS}\smallsetminus K.
\end{cases}
$$
\item Let \(E=\F_{q^2}\) and \(g_1, g_2 \in  E^\times \smallsetminus \F_{q}^\times\).
Let $g \in \mathrm{NS}\cup K$ be such that $\det(g) = \det(g_1 g_2)$.
Then 
$$
m(C_g; g_1, g_2) = \begin{cases} 
q^2-q & \text{if } g \in K\cap Z,\\
1 & \text{if } g \in \mathrm{NS}\cap K,
\\
q+1 & \text{if } g \in \mathrm{NS}\smallsetminus K. 
\end{cases}
$$
\item Let \(g_1\in \mathrm{SS}\) and \(g_2\in \mathrm{NS} \).
Let $g \in \mathrm{NS}\cup \mathrm{SS}$ be such that $\det(g) = \det(g_1 g_2)$.
Then 
$$
m(C_g; g_1, g_2) = \begin{cases} 
q-1 & \text{if } g \in \mathrm{SS},
\\
q+1 & \text{if } g \in \mathrm{NS}. 
\end{cases}
$$
\end{enumerate}
\end{proposition}

\begin{proof}
The proposition follows from Lemma \ref{formula for multiplicity count}, Proposition \ref{main theorem-1-4} and Lemma \ref{lemma:no-scalars-mixed-types}.
\end{proof}

Recall that \([G]=Z\sqcup [\mathrm U]\sqcup [\mathrm{SS}]\sqcup [\mathrm{NS}]\).

\begin{definition}
Let \(G=\GL_2(\F_q)\) and $E=\F_{q}\oplus \F_q$ or $E=\F_{q^2}$. 
For \(g_1,g_2\in E^\times\smallsetminus \F_{q}^\times\), we define the following.
\begin{enumerate}
\item 
\(K:=\{g_1g_2,\,g_1^{\sigma}g_2,\,g_1g_2^{\sigma},\,g_1^{\sigma}g_2^{\sigma}\}\)
and 
\(
\mathbf{K}:=\{[g_1g_2]_{\sigma}, [g_1g_2^{\sigma}]_{\sigma}\}.
\)
\item 
\(\mathbf{SS}  :=\{g\in [\mathrm{SS}] \mid \det(g)=\det(g_1g_2)\}\).
\item 
\(\mathbf{NS} :=\{g\in [\mathrm{NS}] \mid \det(g)=\det(g_1g_2)
\}\). 
\end{enumerate}
\end{definition}
Using Proposition \ref{proposition 3.11-1,2,3}, we get the following corollary.

\begin{corollary}\label{main theorem-***}
\begin{enumerate}
\item Let \(E = \F_{q} \oplus \F_{q}\) and $g_1,g_2\in E^\times \smallsetminus\F_{q}^\times$.
Then
\[
\begin{aligned}
C_{g_1}\cdot C_{g_2}
=& \bigsqcup_{\substack{
g \in \mathbf{K} \cap Z}} (q^2+q)C_{g} 
\bigsqcup_{\substack{
g \in \mathbf{K} \cap \mathbf{SS}
}}(2q-1)C_g
&
\bigsqcup_{\substack{
g\in\mathbf{SS}\smallsetminus \mathbf{K}
}}
(q-1)C_g &
\bigsqcup_{\substack{
g\in [G]\smallsetminus (\mathbf{SS} \cup \mathbf{K})
}}
m(C_{g}; g_1, g_2)C_g.
\end{aligned}
\]
\item Let \(E= \F_{q^2}\) and $g_1,g_2\in E^\times \smallsetminus\F_{q}^\times$.
Then
\[
\begin{aligned}
C_{g_1}\cdot C_{g_2}
=& \bigsqcup_{\substack{
g \in \mathbf{K} \cap Z}} (q^2-q)C_{g} 
\bigsqcup_{\substack{
g \in \mathbf{K} \cap \mathbf{NS}
}}C_g
&
\bigsqcup_{\substack{
g\in\mathbf{NS}\smallsetminus \mathbf{K}
}}
(q+1)C_g &
\bigsqcup_{\substack{
g\in [G] \smallsetminus (\mathbf{NS} \cup \mathbf{K})
}}
m(C_{g}; g_1, g_2)C_g.
\end{aligned}
\]
\item Let \(g_1\in \mathrm{SS}\) and \(g_2\in \mathrm{NS}\).
Then
\[
\begin{aligned}
C_{g_1}\cdot C_{g_2}
= &
\bigsqcup_{\substack{
g\in \mathbf{SS}
}}(q-1)C_g
\bigsqcup_{\substack{
g\in\mathbf{NS}
}}
(q+1)C_g\bigsqcup_{\substack{
g\in [G] \smallsetminus (\mathbf{NS}\cup \mathbf{SS})
}}
m(C_{g}; g_1, g_2)C_g.
\end{aligned}
\]
\end{enumerate}
\end{corollary}

\section{Bijection between $\Irr(G)$ and $\Conj(G)$}\label{sec:6}
In this section, we will define a bijection
\[
\Psi:\Irr(G)\longrightarrow \Conj(G)
\]
for \(G=\GL_2(\F_q)\),
which will allow us to relate the  trilinear forms on representations of \(G\) to the triple product of the corresponding conjugacy classes in \(G\).
Recall the embedding \(E \hookrightarrow M_2(\F_q)\) as in Section \ref{sec:2} for $E=\F_{q}\oplus \F_q$ or $E=\F_{q^2}$ with \(\F_{q} \hookrightarrow E\).
For any finite commutative group \(A\), let \(\widehat{A}\) denote the dual group, i.e. the group of all characters of \(A\).
Since \(\F_{q}^{\times}\) and \(E^{\times}\) are finite groups, we have \(\widehat{\F_{q}^{\times}} \cong \F_{q}^{\times}\) and \(\widehat{E^{\times}} \cong E^{\times}\), but there is no natural isomorphism.
Now we fix isomorphisms \( \widehat{\F_{q}^{\times}} \to \F_{q}^{\times}\) and \( \widehat{E^{\times}} \to E^{\times}\) in a compatible way as in the following lemma.

\begin{lemma}\label{lemma for extension of isomporhism}
Let \(\theta : \widehat{\F_{q}^{\times}} \to \F_{q}^{\times}\) be an isomorphism. 
Then there exists an isomorphism \( \Theta : \widehat{E^{\times}} \to E^{\times}\)
such that for any \( \Lambda \in \widehat{E^{\times}}\) we have \( \theta(\Lambda|_{\F_{q}^{\times}}) = \N(\Theta(\Lambda)) \).
\end{lemma}

\begin{proof}
First, let $E=\F_q\oplus \F_q$. Then \(\widehat{E^{\times}} \cong \widehat{\F_{q}^{\times}} \times \widehat{\F_{q}^{\times}}\) and \(\Theta(\chi_1, \chi_2) = (\theta(\chi_1), \theta(\chi_2))\) defines an isomorphism from \( \widehat{E^{\times}} \rightarrow E^{\times}\). 
Clearly, for \(\Lambda = (\chi_1, \chi_2) \in \widehat{E^{\times}}\), we have \(\Lambda|_{\F_{q}^{\times}} = \chi_1 \chi_2\)
and 
\[
\N(\Theta(\Lambda))
=\theta(\chi_1)\theta(\chi_2)=\theta(\chi_1\chi_2)
=\theta(\Lambda|_{\F_q^\times}).
\]
Now, let $E=\F_{q^2}$.
Then $E^\times$ and $\widehat{E^\times}$ are cyclic groups of order \(q^2-1\). 
If $\widehat{E^\times} = \langle \chi \rangle$ then $\widehat{\F_{q}^\times} = \langle \chi|_{\F_{q}^{\times}} \rangle$ and \( \F_{q}^{\times} = \langle \theta(\chi|_{\F_{q}^{\times}})\rangle\). 
Since \(\N : E^{\times} \to \F_{q}^{\times} \) is surjective, let \(\N(a) = \theta(\chi|_{\F_{q}^{\times}})\) for some \(a \in E^{\times}\). 
However one can chose a suitable \(a \in E^{\times}\) such that \(E^{\times} = \langle a \rangle\) and \(\N(a) = \theta(\chi|_{\F_{q}^{\times}})\).
Now define \( \Theta(\chi)=a\) and extend it to a homomorphism from \(\widehat{E^{\times}} \to E^{\times}\). 
It can be easily checked that \(
\Theta\) satisfies the property stated in the lemma.
\end{proof}

In other words, the lemma says that the following diagram commutes 
\[
\begin{tikzcd}[column sep=huge, row sep=large]
\widehat{E^\times}
    \arrow[r, "\operatorname{Res}"]
    \arrow[d, "\Theta"', "\sim"]
&
\widehat{\mathbb F_q^\times}
    \arrow[d, "\sim"', "\theta"]
\\
E^\times
    \arrow[r, "\N"']
&
\F_{q}^\times
\arrow[ul, phantom, "\circlearrowleft" description]
\end{tikzcd}
\]
where $\Res$ is the restriction map from $\widehat{E^\times}$ to $\widehat{\F_{q}^\times}$.
Since $E$ is either $\F_{q^2}$ or $\F_q\oplus\F_q$, we have that ${\rm Aut}_{\F_{q}}(E)=\langle \sigma \rangle$ is a group of order two. 
Here $\sigma$ is the non-trivial automorphism of $E$ which fixes $\F_q$.
The action of $\sigma$ on $E^{\times}$ 
induces an action on $\widehat{E^\times}$ given by
\(
\sigma(\Lambda)=\Lambda^{\sigma}:=\Lambda\circ \sigma,~\Lambda\in \widehat{E^\times}.
\)
Moreover, we have \(\Theta(\Lambda^{\sigma}) = \Theta(\Lambda)^{\sigma}\).

\begin{definition}
Let $E=\F_{q}\oplus \F_q$ or $E=\F_{q^2}$.
\begin{enumerate}
\item Define an equivalence relation \(\sim \) on $\widehat{E^\times}$ by
\( \Lambda\sim\Lambda' \)
if
\( \Lambda'=\Lambda \text{ or } \Lambda'=\Lambda^\sigma \).
\item A character \( \Lambda \in \widehat{E^\times} \) is called decomposable if \( \Lambda = \chi \circ \mathrm{N}\) for some \( \chi \in \widehat{\F_{q}^{\times}}\).  
\item A character \( \Lambda \in \widehat{E^\times} \) is called indecomposable if it is not decomposable.
\end{enumerate}
\end{definition}

\begin{lemma}\label{lemma relating non-split class and cuspidal representations}
Let $E$ be either $\F_q\oplus\F_q$ or $\F_{q^2}$.
Then
\(\Lambda \in \widehat{E^{\times}}\) is decomposable if and only if \(\Theta(\Lambda) \in \F_q^\times\).
\end{lemma}

\begin{proof}
Let \(\Lambda  \in \widehat{E^{\times}}\) be  decomposable, i.e, \(\Lambda = \chi \circ \N \) for some \(\chi \in \widehat{\F_{q}^{\times}}\).
Then 
\(\Lambda^{\sigma} = \Lambda\) and hence \(\Theta(\Lambda)^{\sigma} = \Theta(\Lambda)\), i.e. \(\Theta(\Lambda) \in \F_{q}^{\times}\).
Conversely, let 
$\Theta(\Lambda)\in \F_{q}^\times$. 
Then $\Theta(\Lambda)=\Theta(\Lambda)^{\sigma}=\Theta(\Lambda^{\sigma})$.
Since $\Theta$ is an isomorphism, we get $\Lambda=\Lambda^{\sigma}$.
By Hilbert's Theorem 90, if $x \in \ker(\N)$ then \( x=y\sigma(y)^{-1} \) for some \(y \in E^{\times}\).
Then $\Lambda(x)=\Lambda(y\sigma(y)^{-1})=\Lambda(y)\Lambda^{\sigma}(y)^{-1} =1$ as $\Lambda^{\sigma}=\Lambda$.
This gives
\(
\ker(\N)\subset \ker(\Lambda)
\)
and hence $\Lambda$ factors through the norm map, i.e.
\( \Lambda=\chi\circ \N \) for some $\chi\in \widehat{\F_{q}^\times}$.
This proves the lemma.
\end{proof}

Now we summarize the irreducible representations of \(\GL_2(\F_q)\)  in the following proposition, see \cite[Section 5.2]{Fulton-Harris} for more details.

\begin{proposition}
\label{classification of irreducible representations of G}
If \(\pi \in \Irr(\GL_2(\F_q)) \) then \(\dim (\pi) \in \{ 1, q, q+1, q-1 \}\).
\begin{enumerate}
\item The 1-dimensional representations are  \( \chi^{1}_{\alpha} :=\alpha \circ \det\) for some \(\alpha \in \widehat{\F_{q}^{\times}}\).
\item The $q$-dimensional representations are  \(\chi_{\alpha}^{q} := \mathrm{St} \otimes \chi^{1}_{\alpha}\) for some \(\alpha \in \widehat{\F_{q}^{\times}}\), where \( \mathrm{St}\) is the Steinberg representation of \(\GL_2(\F_q)\).
\item The $(q+1)$-dimensional representations are \( \chi^{q+1}_{\Lambda} := \Ind_{B}^{G}(\Lambda)\) the irreducible principal series representations associated to an indecomposable \( \Lambda \in \widehat{E^{\times}}\) with \(E= \F_{q} \oplus \F_q\). Moreover, \(\chi^{q+1}_{\Lambda_1} \cong \chi^{q+1}_{\Lambda_2} \iff \Lambda_1 \sim \Lambda_2\). 
\item The $(q-1)$-dimensional representations are  \( \chi^{q-1}_{\Lambda}\) the irreducible cuspidal representations associated to indecomposable \( \Lambda \in \widehat{E^{\times}}\) with \(E= \F_{q^2}\).
Moreover, \(\chi^{q-1}_{\Lambda_1} \cong \chi^{q-1}_{\Lambda_2} \iff \Lambda_1 \sim \Lambda_2\). 
\end{enumerate}
\end{proposition}

\begin{definition}\label{definition of Psi map}
Let $E$ be either $\F_{q}\oplus \F_q$ or $\F_{q^2}$.
We fix $\theta : \widehat{\F_{q}^{\times}} \rightarrow \F_{q}^{\times}$ and  $\Theta : \widehat{E^\times} \rightarrow E^{\times}$ as in Lemma \ref{lemma for extension of isomporhism}.
For \(G = \GL_2(\F_q)\), define 
\begin{align}\label{definition  of Psi}
\Psi:\Irr(G)\longrightarrow \Conj(G)
\end{align}
as follows:
\begin{enumerate}
\item  \( \Psi\!\left(\chi_{\alpha}^{\,1}\right) = C_{\theta(\alpha)}\) for $\alpha \in \widehat{\F_{q}^{\times}}$.
\item  \( \Psi\!\left(\chi_{\alpha}^{\,q}\right) = C_{\theta(\alpha)u}\) for $\alpha \in \widehat{\F_{q}^{\times}}$ and \(u = \begin{pmatrix}
1&1\\0&1    
\end{pmatrix}\).
\item  \(\Psi (\chi^{q+1}_{\Lambda}) = C_{\Theta(\Lambda)}\) for indecomposable $\Lambda \in \widehat{E^{\times}}$ with \(E= \F_q \oplus \F_q\).
\item  \(\Psi (\chi^{q-1}_{\Lambda}) = C_{\Theta(\Lambda)}\) for indecomposable $\Lambda \in \widehat{E^{\times}}$ with \(E= \F_{q^2} \).
\end{enumerate}
\end{definition}

\section{Triple products and Trilinear forms} \label{Sec-5}

\subsection{Triple product of conjugacy classes}
For $g_1, g_2, g_3 \in G$, in this subsection, we compute 
\(
m\bigl(I,C_{g_1}\cdot C_{g_2}\cdot C_{g_3}\bigr)
\) the multiplicity of the identity matrix $I$ in the triple product
$C_{g_1}\cdot C_{g_2}\cdot C_{g_3}$ of conjugacy classes, viewed as a multi set.

\begin{lemma}\label{existence of I in triple product always}
Let \(g_1, g_2, g_3 \in \mathrm{SS} \cup \NS\) be such that \(\det(g_1 g_2 g_3)=1\). Then \(I \in C_{g_1}\cdot C_{g_2}\cdot C_{g_3}\).  
\end{lemma}

\begin{proof}
First, assume that all \(g_i\in \mathrm{SS}\) for \(i=1,2,3\).
We take \(g_i\in E^\times\smallsetminus \F_{q}^\times\), where
\(E=\F_q\oplus \F_q\).
Using the description of \(C_{g_1}\cdot C_{g_2}\) in Proposition \ref{main theorem-***} (1) together with Lemma \ref{lemma:no-scalars-mixed-types}, we obtain
\(I\in C_{g_1}\cdot C_{g_2}\cdot C_{g_3}
\iff
I\in C_g\cdot C_{g_3}\)
for a unique \(g\in \mathbf{SS}\).
Since
\(\det(g_1g_2g_3)=1\) and \(g_3\in \mathrm{SS}\),
we have
\([g_3^{-1}]_{\sigma}\in \mathbf{SS}
\).
Thus, 
\(I\in C_{g_1}\cdot C_{g_2}\cdot C_{g_3}\).
Now, the remaining cases follow by a similar argument.
\end{proof}

\begin{remark} \label{Horn's problem}
The above lemma gives an affirmative answer to multiplicative Horn's problem or Deligne-Simpson problem (see \cite{Simpson-1990, Arielle-2025}) for conjugacy classes of non-scalar semisimple elements in \(G=\GL_2(\F_q)\). This problem asks if there exists \((x,y,z) \in C_{g_1} \times  C_{g_2} \times C_{g_3} \) such that \(xyz=I\). 
What becomes clear is that the necessary condition \(\det(g_1 g_2 g_3)=1\) is also sufficient for the existence of \((x,y,z)\) such that \(xyz=I\), i.e.  \( I \in C_{g_1}\cdot C_{g_2}\cdot C_{g_3} \).
Although, in Lemma \ref{existence of I in triple product always}, it is assumed that all \(g_i\)'s are not scalar and semisimple, but it is not difficult to verify that the assumption of \(g_i\)'s being semisimple is not necessary.
\end{remark}
The following proposition gives \( m(I, C_{g_1}\cdot C_{g_2}\cdot C_{g_3})\) the number  of triples \((x,y,z) \in C_{g_1} \times  C_{g_2} \times C_{g_3} \) satisfying \(xyz=I\).

\begin{proposition}
\label{triple product of non split conjugacy classes}
Let \(g_1, g_2, g_3 \in \mathrm{SS}  \cup \NS\) be such that \(\det(g_1 g_2 g_3)=1\). 
For $g_1,g_2,g_3\in E^\times \smallsetminus \F_{q}^{\times}$ with \(E= \F_q \oplus \F_q\) or \(E=\F_{q^2}\), let \( \mathcal{A} :=\{g'_1g'_2g'_3:g'_i\in \{g_i, g_i^{\sigma}\}\} \).
Then
\begin{enumerate}
\item If all \(g_i \in \mathrm{SS}\), we take  \(g_1, g_2, g_3 \in E^{\times} \smallsetminus \F_{q}^{\times}\) with \( E=\F_{q}\oplus \F_q\), then \\ 
\(
m(I, C_{g_1}\cdot C_{g_2}\cdot C_{g_3})=\begin{cases}
(2q-1)(q^2+q) & \text{ if $I\in \mathcal{A}$},\\
q^3-q & \text{ otherwise}.
\end{cases}
\)
\item If all \(g_i \in \NS\), we take  \(g_1, g_2, g_3 \in E^{\times} \smallsetminus \F_{q}^{\times}\) with \( E=\F_{q^2}\), then \\
\(
m(I, C_{g_1}\cdot C_{g_2}\cdot C_{g_3})=\begin{cases}
q^2-q & \text{ if $I\in \mathcal{A}$},\\
q^3-q & \text{ otherwise}.
\end{cases}
\)
\item If the \(g_i\)'s are not all in \(\mathrm{SS}\) or not all in \(\mathrm{NS}\), then
\(
m\left(I,C_{g_1}\cdot C_{g_2}\cdot C_{g_3}\right)=q^3-q.
\)
\end{enumerate}
\end{proposition}

\begin{proof}
Recall that \(K=\{g_1g_2, g_1g_2^{\sigma}, g_1^{\sigma}g_2, g_1^{\sigma}g_2^{\sigma}\}\) and 
\(
\mathbf{K}:=\{[g_1g_2]_{\sigma},[g_1g_2^{\sigma}]_{\sigma}\}.
\)
Then 
\[
I\in \mathcal{A} \iff [g_3^{-1}]_{\sigma}\in \mathbf{K}.
\]
First, we consider all \(g_i \in \mathrm{SS}\).
By Lemma \ref{existence of I in triple product always},  we have $[g_3^{-1}]_{\sigma}\in \mathbf{SS}$
and then \(I \in C_{g_3^{-1}} \cdot C_{g_3}\) with \(m(I;g_3^{-1},g_3)=|C_{g_3}|=q^2+q\) by Lemma \ref{lemma:no-scalars-mixed-types}.
Now, using Corollary \ref{main theorem-***} (1), if \(I \in \mathcal{A}\) then \([g_3^{-1}
]_{\sigma} \in \mathbf{SS} \cap \mathbf{K}\) with \(m(g_3^{-1}; g_1, g_2) = 2q-1\), otherwise \([g_3^{-1}]_{\sigma} \in \mathbf{SS} \smallsetminus \mathbf{K}\) with \(m(g_3^{-1}; g_1, g_2) = q-1\). 
Thus, \( m(I,C_{g_1}\cdot C_{g_2}\cdot C_{g_3})\) equals  \( (2q-1)(q^2+q)\) if \(I \in \mathcal{A}\) and \((q-1)(q^2+q)\) otherwise.
This proves part (1) of the proposition. Parts (2) and (3) are similar, which we skip.
\qedhere
\end{proof}

\subsection{Trilinear forms}
For $\pi_i\in\Irr(G)$, $i=1,2,3$, in this subsection, we compute
\(
m(\mathbb{1},\pi_1\otimes\pi_2\otimes\pi_3)
\)
the multiplicity of the trivial representation $\mathbb{1}$ in
$\pi_1\otimes\pi_2\otimes\pi_3$, viewed as a representation of $G$
via the diagonal embedding
\(
G\hookrightarrow G\times G\times G.
\)
For $\Lambda\in \widehat{E^\times}$, let $\overline{\Lambda}:=\Lambda|_{\F_q^\times}$.
For $\Lambda_1, \Lambda_2\in \widehat{E^\times}$, we define
\(K':=\{\Lambda_1\Lambda_2,\Lambda_1^{\sigma}
\Lambda_2, \Lambda_1\Lambda_2^{\sigma}
, \Lambda_1^{\sigma}
\Lambda_2^{\sigma}\}
 \).
The following lemma is a consequence of \cite[Theorem 1.1]{Archita-Hassain}.

\begin{lemma}\label{existenceof trivial character in tensor products}
\begin{enumerate}
\item For $\pi_1,\pi_2\in \Irr(G)$, we have \(m(\mathbb{1}, \pi_1 \otimes \pi_2) \leq 1\).
\item For $\pi_1\in \Irr(G)$, $\exists$ a unique $\pi_2\in \Irr(G)$ such that $m(\mathbb{1}, \pi_1 \otimes \pi_2)= 1$.
\item Let $(\pi_1, \pi_2)=(\chi^{q+1}_{\Lambda_1},\chi^{q+1}_{\Lambda_2})$ or $(\pi_1, \pi_2)=(\chi^{q-1}_{\Lambda_1},\chi^{q-1}_{\Lambda_2})$ for indecomposable \(\Lambda_1, \Lambda_2 \in \widehat{E^{\times}}\).
Then \(m(\mathbb{1}, \pi_1 \otimes \pi_2) = 1 \iff  \mathbb{1} \in K'\).
\end{enumerate} 
\end{lemma}

\begin{proof}
We know that \(
\begin{aligned}
\Hom_G(\pi_1\otimes\pi_2, \mathbb{1}) \cong \Hom_G(\pi_1,\pi_2^\vee),
\end{aligned}
\)
where  \(\pi_2^\vee\) is the dual representation of $\pi_2$.
Now, parts (1) and (2) follow from Schur's lemma.
Part (3) follows from 
\(
\chi_{\Lambda}^{\varepsilon}
\cong 
\chi_{\Lambda^\sigma}^{\varepsilon}
\)
and
\(
\left(\chi_{\Lambda}^{\varepsilon}\right)^\vee
\cong 
\chi_{\Lambda^{-1}}^{\varepsilon}
\) 
for $\varepsilon\in\{q+1,q-1\}$, 
see \cite{Fulton-Harris}.
\qedhere
\end{proof}

Let $E=\F_{q}\oplus \F_q$ or $E=\F_{q^2}$.
For indecomposable $\Lambda\in \widehat{E^\times}$, define 
\(
[\Lambda]_{\sigma}:=\{\Lambda, \Lambda^{\sigma}\},~\widehat{E^\times}/\langle \sigma\rangle:=\{[\Lambda]_{\sigma}:\Lambda\in \widehat{E^\times},\Lambda \text{ indecomposable} \}.
\)

\begin{proposition}
\label{trivial representation in pi_1 otimes pi_2 otimes pi_3}
Let \(E\) be either \(\F_q\oplus \F_q\) or \(\F_{q^2}\).
For \(i=1,2,3\), let \(\Lambda_i\in\widehat{E^\times}\) be indecomposable such that \(\overline{\Lambda}_1\overline{\Lambda}_2\overline{\Lambda}_3=\mathbb{1}\).
If \(\Lambda_1,\Lambda_2,\Lambda_3\) are characters of the same $E^\times$, then define
\(
\mathcal{B}
:=
\{\Lambda'_1\Lambda'_2\Lambda'_3:
\Lambda'_i\in\{\Lambda_i,\Lambda_i^\sigma\}\}.
\)
Then
\begin{enumerate}
\item For \(i=1,2,3\), if  \(\pi_i=\chi_{\Lambda_i}^{q+1}\), then
\(
m(\mathbb{1},\pi_1\otimes\pi_2\otimes\pi_3)
=
\begin{cases}
2 & \text{if \(\mathbb{1}\in\mathcal{B}\)},\\
1 & \text{otherwise}.
\end{cases}
\)

\item For \(i=1,2,3\), if \(\pi_i=\chi_{\Lambda_i}^{q-1}\), then
\(
m(\mathbb{1},\pi_1\otimes\pi_2\otimes\pi_3)
=
\begin{cases}
0 & \text{if \(\mathbb{1}\in\mathcal{B}\)},\\
1 & \text{otherwise}.
\end{cases}
\)

\item If \(\pi_i=\chi^{q+1}_{\Lambda_i}\) or  \(\pi_i=\chi^{q-1}_{\Lambda_i}\) and not all have the same dimension, then
\[
m\left(\mathbb{1},
\pi_1\otimes\pi_2\otimes\pi_3\right)=1.
\]
\end{enumerate}
\end{proposition}

\begin{proof}
\begin{enumerate}
\item Let $\pi_i=\chi^{q+1}_{\Lambda_i}$ for \(i=1,2,3\).
By \cite[Theorem~1.1 (9)]{Archita-Hassain}, we have
\begin{align}\label{eq-6.8}
\pi_1\otimes\pi_2
=
& \underset{[\Gamma]_{\sigma}\in S_{\overline{\Lambda}_1\overline{\Lambda}_2}}{\bigoplus}  \chi^{q+1}_{\Gamma}\bigoplus \Ind_B^{G}(\Lambda_1\Lambda_2)\otimes \pi_3
\bigoplus \Ind_B^{G}(\Lambda_1\Lambda_2^{\sigma})\otimes \pi_3
\nonumber
\\
& \bigoplus ~(\text{non-principal series constituents}),
\end{align}
where $S_{\overline{\Lambda}_1\overline{\Lambda}_2}:=\{[\Gamma]_{\sigma}\in \widehat{E^\times}/\langle \sigma\rangle :~\Gamma|_{\F_q^\times}=\overline{\Lambda}_1\overline{\Lambda}_2\}$.
Note that if $\pi$ is not a principal series representation of $G$, then  \(m(\mathbb{1}, \pi\otimes \pi_3)=0\).
Since  $\overline{\Lambda}_1\overline{\Lambda}_2\overline{\Lambda}_3=\mathbb{1}$, using Lemma \ref{existenceof trivial character in tensor products}, we have a unique 
\(
[\Gamma]_{\sigma}=[\Lambda_3^{-1}]_{\sigma}\in S_{{\overline{\Lambda}_1\overline{\Lambda}_2}}
\)
such that $m(\mathbb{1},\chi_{\Gamma}^{q+1}\otimes \pi_3)=1$.
Now, using Lemma \ref{existenceof trivial character in tensor products}, it can be checked that
\[
\mathbb{1}\in \Ind_B^{G}(\Lambda_1\Lambda_2)\otimes \chi^{q+1}_{\Lambda_3}
\bigoplus \Ind_B^{G}(\Lambda_1\Lambda_2^{\sigma})\otimes \chi^{q+1}_{\Lambda_3} \iff \mathbb{1}\in \mathcal{B}.
\]
This proves (1).

\item Let $\pi_i=\chi^{q-1}_{\Lambda_i}$ for \(i=1,2,3\).
By \cite[Theorem~1.1(10)--(13)]{Archita-Hassain}, we have
\begin{align}\label{tensor product of two cuspidals in theorem 8.15}
\pi_1\otimes\pi_2
=
\bigoplus_{[\Gamma]_{\sigma}\in
W_{\overline{\Lambda}_1\overline{\Lambda}_2}\smallsetminus X}
\chi^{q-1}_{\Gamma}
\bigoplus
(\text{non-cuspidal constituents}),
\end{align}
where
\(
W_{\overline{\Lambda}_1\overline{\Lambda}_2}
:=
\left\{
[\Gamma]_{\sigma}\in
\widehat{E^\times}/\langle\sigma\rangle:
\Gamma|_{\mathbb F_q^\times}
=
\overline{\Lambda}_1\overline{\Lambda}_2
\right\}\) and 
\(
X:=
\left\{
[\Lambda_1\Lambda_2]_{\sigma},
[\Lambda_1\Lambda_2^{\sigma}]_{\sigma}
\right\}
\cap W_{\overline{\Lambda}_1\overline{\Lambda}_2}.
\)
Now, consider $\pi_1\otimes\pi_2\otimes\pi_3$.
Note that if $\pi$ is not a cuspidal representation of $G$, then  \(m(\mathbb{1}, \pi\otimes \pi_3)=0\).
Since
\(
\overline{\Lambda}_1\overline{\Lambda}_2\overline{\Lambda}_3
=
\mathbb{1},
\) by Lemma \ref{existenceof trivial character in tensor products} (2), there exists a unique
\([\Gamma]_{\sigma} :=[\Lambda_3^{-1}]_{\sigma}\in W_{\overline{\Lambda}_1\overline{\Lambda}_2}\) such that
\(
m\bigl(\mathbb{1},
\chi_{\Gamma}^{q-1}\otimes\chi_{\Lambda_3}^{q-1}\bigr)=1.
\)
Now, the proof follows by observing that 
\(
[\Lambda_3^{-1}]_{\sigma}\in X
\iff
\mathbb{1}\in\mathcal{B}.
\)
\item Let \(\pi_1=\chi^{q+1}_{\Lambda_1}\) and \( \pi_2=
\chi^{q-1}_{\Lambda_2}\).
Then using \cite[Theorem 1.1 (9)]{Archita-Hassain}, we have 
\begin{align}\label{tensor product of split and non split in theorem 5.5}
\pi_1\otimes\pi_2
=
\bigoplus_{[\Gamma]_{\sigma}\in
W_{\overline{\Lambda}_1\overline{\Lambda}_2}}
\chi^{q-1}_{\Gamma}\bigoplus_{[\Gamma']_{\sigma}\in
S_{\overline{\Lambda}_1\overline{\Lambda}_2}}
\chi^{q+1}_{\Gamma'}
\bigoplus
\text{ (other  constituents)},
\end{align}
where \(W_{\overline{\Lambda}_1\overline{\Lambda}_2}\) and 
\(S_{\overline{\Lambda}_1\overline{\Lambda}_2}\) are as before.
Now it can be easily  seen that for  \( \pi_3 = \chi^{q+1}_{\Lambda_3} \) or \( \pi_3 = \chi^{q-1}_{\Lambda_3} \), we get \(m( \mathbb{1}, \pi_1 \otimes \pi_2\otimes \pi_3) =1\).
\qedhere
\end{enumerate}    
\end{proof}

\subsection{Relating trilinear forms with triple product of conjugacy classes}
In this subsection, we relate the trilinear forms with the triple product of conjugacy classes and prove the main theorem of the paper.
\begin{definition}\label{exceptional}
\begin{enumerate}
\item Let \( \pi_i = \chi_{\Lambda_i}^{q+1} \) for \(i=1,2,3\) or \( \pi_i = \chi_{\Lambda_i}^{q-1} \) for \(i=1,2,3\). We say that \(\{ \pi_1, \pi_2, \pi_3 \}\) is exceptional if 
\(
\mathbb{1} \in \{ \Lambda_1 ' \Lambda_2' \Lambda_3 ' : \Lambda_i ' \in \{ \Lambda_i, \Lambda_i^{\sigma} \} \} 
\).
\item Let \( g_i \in \F_q^{\times} \times \F_{q}^{\times}\smallsetminus\F_{q}^\times\) for \(i =1,2,3\) or \( g_i \in \F_{q^2}^{\times}\smallsetminus\F_{q}^\times \) for \(i =1,2,3\). We say that \(\{C_{g_1}, C_{g_2}, C_{g_3} \}\) is exceptional if 
\(
I \in \{ g_1 ' g_2 ' g_3 ' : g_i ' \in \{ g_i, g_i^{\sigma}\} \}.
\)
\end{enumerate}
\end{definition}

\begin{theorem} \label{main theorem}
Let \(\pi_i\) for \(i=1,2,3\) be irreducible Deligne-Lusztig representations of $\GL_2(\F_q)$ such that the product of their central characters is trivial.
Let \( \Psi(\pi_i) =C_{g_i}\) for \(i=1,2,3\). 
Then 
\begin{enumerate}
\item For \(i=1,2,3\), let \( \pi_i = \chi_{\Lambda_i}^{q+1} \)  or \( \pi_i = \chi_{\Lambda_i}^{q-1} \).
Then \(\{ \pi_1, \pi_2, \pi_3 \}\) is exceptional if and only if \(\{C_{g_1}, C_{g_2}, C_{g_3} \}\) is exceptional.
\item \( m(\mathbb{1}, \pi_1 \otimes \pi_2 \otimes \pi_3 ) = 1 \iff m( I, C_{g_1} \cdot C_{g_2} \cdot C_{g_3}) =q^3 -q\).
\item Let \( \pi_i = \chi_{\Lambda_i}^{q+1} \) for \(i=1,2,3\).
Then \( m(\mathbb{1}, \pi_1 \otimes \pi_2 \otimes \pi_3) = 2 \iff m(I, C_{g_1} \cdot C_{g_2} \cdot C_{g_3}) =(2q-1)(q^2 +q)\).
In this case, \(\{\pi_1, \pi_2, \pi_3\}\) is exceptional.
\item Let \( \pi_i = \chi_{\Lambda_i}^{q-1} \) for \(i=1,2,3\).
Then \( m(\mathbb{1}, \pi_1 \otimes \pi_2 \otimes \pi_3) = 0 \iff m(I, C_{g_1} \cdot C_{g_2} \cdot C_{g_3}) =q^2 -q \).
In this case, \(\{\pi_1, \pi_2,\pi_3\}\) is exceptional.
\end{enumerate}
\end{theorem}

\begin{proof}
Using Lemma \ref{lemma for extension of isomporhism},  $\Theta(\Lambda_i)=g_i$ and
$\Theta(\Lambda_i^\sigma)=g_i^\sigma$.
We get
$\Theta(\mathcal{B})= \mathcal{A}$, where $\mathcal{A}$ and $\mathcal{B}$ be as in Propositions \ref{triple product of non split conjugacy classes}, \ref{trivial representation in pi_1 otimes pi_2 otimes pi_3}.
This proves (1).
Part (2), (3) and (4) follow from Propositions \ref{triple product of non split conjugacy classes}, \ref{trivial representation in pi_1 otimes pi_2 otimes pi_3}.
\qedhere
\end{proof}

The above theorem proves part (1) of Theorem \ref{intro: main theorem} in the Introduction.

\begin{definition}\label{definition of indicators}
For given \( \pi_1, \pi_2, \pi_3 \in \Irr(\GL_2(\F_q))\), let
\(
\mathcal{E}_{ss}
\)
(respectively, \(\mathcal{E}_{ns}\)) equal 1 if
$\{\pi_1,\pi_2,\pi_3\}$ is exceptional with each
$\pi_i$ an irreducible principal series representation (respectively, cuspidal representation), and
$0$ otherwise.
\end{definition}

\begin{corollary}\label{last corollary}
Let $\pi_1,\pi_2,\pi_3$ be irreducible Deligne--Lusztig
representations of $\GL_2(\F_q)$ and let $\Psi(\pi_i)=C_{g_i}$ for $i=1,2,3$.
Assume that the product of the central characters of $\pi_1,\pi_2,\pi_3$ is trivial.
Then
\[
m\bigl(I,C_{g_1}\cdot C_{g_2}\cdot C_{g_3}\bigr)
=
q(q^2-1)\,
m\bigl(\mathbb{1},\pi_1\otimes\pi_2\otimes\pi_3\bigr)
+
q(q+1)\mathcal{E}_{ss}
+
q(q-1)\mathcal{E}_{ns}.
\]
\end{corollary}

Theorem~\ref{intro: main theorem}  and the table written therein follows from Corollary~\ref{last corollary}.

\begin{remark}
\begin{enumerate}
\item 
It is also interesting to note that except for the exceptional cases, we have $m(I,C_{g_1}\cdot C_{g_2}\cdot C_{g_3})=q^3-q=|\mathrm{PGL}_2(\F_q)|$.
Analogously, for $G=\GL_n(\F_q)$, except for some exceptional cases, we expect 
\[
m(I,C_{g_1}\cdot C_{g_2}\cdot C_{g_3})=|\mathrm{PGL}_n(\F_q)|=q^{n^2-1}\prod_{i=2}^{n}(1-q^{-i}).
\] 
\item
Observe that
\(
m\bigl(\mathbb{1},\pi_1\otimes\pi_2\otimes\pi_3\bigr)
\)
is the coefficient of $q^3$ in
\(
m\bigl(I,C_{g_1}\cdot C_{g_2}\cdot C_{g_3}\bigr),
\)
viewed as a polynomial in $q$.
Analogously, for $G=\GL_n(\F_q)$ we expect that \(
m\bigl(\mathbb{1},\pi_1\otimes\pi_2\otimes\pi_3\bigr)
\) is the coefficient of \( q^{n^2-1} \)
in \(
m\bigl(I,C_{g_1}\cdot C_{g_2}\cdot C_{g_3}\bigr)
\) seen as a polynomial in \(q\).
\end{enumerate}
\end{remark}

\section*{Acknowledgements}
The authors thank Dipendra Prasad for suggesting the problem and several insightful comments. 
AP acknowledges the support of IIT Bombay through the Institute Postdoctoral Fellowship (20004759).

\bibliographystyle{abbrv}
\bibliography{references}

\end{document}